\documentclass[preprint]{elsarticle}

\usepackage{amsmath,amsfonts,amssymb,amsthm}
\usepackage[english]{babel}

\newtheorem{theorem}{Theorem}[section]
\newtheorem{corollary}{Corollary}[section]
\newtheorem{proposition}{Proposition}[section]

\newtheorem{definition}{Definition}[section]
\newtheorem{example}{Example}[section]
\renewcommand{\labelenumi}{(\theenumi)}

\begin{document}

\title{On $\alpha$-embedded subsets of products}
\tnotetext[t1]{The paper was submitted on August, 23, 2013 to the Central European Journal of Mathematics, but in spite of positive referee's report
 the manuscript was withdrawn due to financial demand for publication and resubmitted to the European Journal of Mathematics}

\author{Olena Karlova\corref{cor1}}\ead{maslenizza.ua@gmail.com}

\author{Volodymyr Mykhaylyuk}\ead{vmykhaylyuk@ukr.net}

\cortext[cor1]{Corresponding author}

\address{Department of Mathematical Analysis, Faculty of Mathematics and Informatics, Chernivtsi National University,
Kotsyubyns'koho str., 2, Chernivtsi, 58012, Ukraine}


\begin{abstract}
We prove that every continuous function $f:E\to Y$ depends on countably many coordinates, if $E$ is an $(\aleph_1,\aleph_0)$-invariant pseudo-$\aleph_1$-compact subspace of a product of topologi\-cal spaces and $Y$ is a space with a regular $G_\delta$-diagonal. Using this fact for any  $\alpha<\omega_1$ we construct an $(\alpha+1)$-embedded subspace of a completely regular space which is not $\alpha$-embedded.
\end{abstract}

\maketitle

\section{Introduction}

If $P$ is a property of functions, then by $P(X)$ ($P^*(X)$) we denote the collection of all real-valued (bounded) functions on a topological space $X$ with the property $P$. By the symbol $C$ we denote the property of continuity and let $B_\alpha$ be the property of being the function of the $\alpha$-th Baire class, where $0\le\alpha<\omega_1$.

Recall that a subset $A$ of a space $X$ is {\it functionally closed (open) in $X$}, if there is $f\in C^*(X)$ with $A=f^{-1}(0)$ ($A=X\setminus f^{-1}(0)$).

The system of all functionally open (closed) subsets of a space $X$ we denote by ${\mathcal G}_0^*$ (${\mathcal
F}_0^*$). Assume that the classes ${\mathcal G}_\xi^*$ and ${\mathcal F}_\xi^*$ are defined for all $\xi<\alpha$, where
$0<\alpha<\omega_1$. Then, if $\alpha$ is odd, the class ${\mathcal G}_\alpha^*$ (${\mathcal F}_\alpha^*$) consists of all countable intersections (unions) of sets of lower classes, and, if $\alpha$ is even the class  ${\mathcal G}_\alpha^*$ (${\mathcal F}_\alpha^*$) consists of all countable unions (intersections) of sets of lower classes. The classes ${\mathcal F}_\alpha^*$ for odd $\alpha$ and ${\mathcal G}_\alpha^*$ for even $\alpha$ are said to be
{\it functionally additive}, and the classes ${\mathcal F}_\alpha^*$ for even $\alpha$ and ${\mathcal G}_\alpha^*$ for odd $\alpha$ are called {\it functionally multiplicative}. A set $A$ is {\it functionally measurable}, if $A\in \bigcup\limits_{0\le\alpha<\omega_1}({\mathcal F}_\alpha^*\cup {\mathcal G}_\alpha^*)$. Notice that the $\sigma$-algebra of functionally measurable subsets of $X$ is also called the $\sigma$-algebra of Baire sets.

An important role in the extension theory play $z$-embedded sets (a subset $A$ of a topological space $X$ is {\it $z$-embedded} in $X$, if for any functionally closed set $F$ in $A$ there exists a functionally closed set $B$ in $X$ such that $B\cap A=F$). In \cite{KarlovaCMUC} for any $\alpha<\omega_1$ it was introduced the notion of an $\alpha$-embedded set, i.e. such a set $A\subseteq X$ that every its subset $B$ of the functionally multiplicative class $\alpha$ in $A$ is the restriction on $A$ of some set of the functionally multiplicative class $\alpha$ in $X$. Obviously, the class of $0$-embedded sets coincides with the class of  $z$-embedded sets. It is not hard to verify that any $\alpha$-embedded set is $\beta$-embedded, if $\alpha\le \beta$ \cite[Proposition 2.5]{KarlovaCMUC}. The converse proposition is not true as Theorem~2.6 from \cite{KarlovaCMUC} shows. In that theorem there was constructed the example of the $1$-embedded subset $E$ of the product $X=[0,1]\times\prod\limits_{t\in [0,1]}X_t$, where $X_t=\mathbb N$ for all $t\in [0,1]$, which is not $0$-embedded in $X$. In the given paper we generalize the above-mentioned result from \cite{KarlovaCMUC} and show that for any $\alpha<\omega_1$ there exists a set $E\subseteq X$, which is $(\alpha+1)$-embedded and is not $\alpha$-embedded in $X$.

The convenient tool in the investigation of properties of an $\alpha$-embedded subset $E$ of a product $\prod\limits_{t\in T}X_t$ is the fact that under some conditions on $E$ every continuous function $f:E\to\mathbb R$ depends on countably many coordinates (see definitions in Section~\ref{sec:def_and_aleph_comp}). S.~Mazur in~\cite{Masur} introduced invariant sets under projection (see Definition~\ref{def:invariant}(a)) and proved that every continuous function $f:E\to Y$ depends on countably many coordinates, if $E\subseteq \Sigma(a)$  for some $a\in E$ and $E$ is invariant under projection, $X_t$ is a metrizable separable space for each $t\in T$ and $Y$ is a Hausdorff space with a $G_\delta$-diagonal. R.~Engelking~\cite{Eng1} showed the same in the case, when $E$ is s set which is invariant under composition (see Definition~\ref{def:invariant}(b)) which is contained in $\Sigma(a)$ for some $a\in E$, $X_t$ is a $T_1$-space with countable base for each $t\in T$ and $Y$ is a Hausdorff space in which every one-point set is $G_\delta$ (see also~\cite{Hus}). N.~Noble and M.~Ulmer~\cite{NobleUlmer} obtained the dependence  on countable many coordinates of a continuous function $f:E\to Y$, if $E$ is a subset of a pseudo-$\aleph_1$-compact space $\prod\limits_{t\in T}X_t$, which contains $\sigma(a)$ for some $a\in E$ and $Y$ is a space with a regular $G_\delta$-diagonal. The result of Noble and Ulmer was generalized by W.W.~Comfort and I.S.~Gotchev in~\cite{ComfG}. Here we consider the so-called $(\aleph_1,\aleph_0)$-invariant subsets of products and, developing the methods of Mazur and of Noble and Ulmer, we show that every continuous function $f:E\to Y$ depends on countably many coordinates if $E$ is an $(\aleph_1,\aleph_0)$-invariant pseudo-$\aleph_1$-compact subspace of a product of topological spaces $X_t$  and $Y$ is a space with a regular $G_\delta$-diagonal.

\section{Some properties of pseudo-$\aleph_1$-compact invariant sets}\label{sec:def_and_aleph_comp}

Let $(X_t:t\in T)$ be a family of non-empty topological spaces, $X=\prod\limits_{t\in T}X_t$ and let $a=(a_t)_{t\in T}$ be a fixed point of $X$. For  $S\subseteq T$ we denote by $p_S$ the projection \mbox{$p_S:X\to\prod\limits_{t\in S}X_t$},  where $p_S(x)=(x_t)_{t\in S}$ for each $x=(x_t)_{t\in T}\in X$; by $x_S^a$ we denote the point with the coordinates $(y_t)_{t\in T}$, where $y_t=x_t$, if $t\in S$ and $y_t=a_t$, if $t\in T\setminus S$. For a basic open set $U=\prod\limits_{t\in T}U_t\subseteq X$ let $N(U)=\{t\in T: U_t\ne X_t\}$.

\begin{definition}\label{def:invariant}
  {\rm A set $E\subseteq X$ is called
  \renewcommand{\theenumi}{\alph{enumi}}
  \renewcommand{\labelenumi}{\theenumi)}
  \begin{enumerate}
    \item {\it invariant under projection} \cite{Masur}, if  $x_S^a\in E$ for any $x\in E$ and $S\subseteq T$;

    \item {\it invariant under composition} \cite{Eng1}, if for any $x,y\in E$ and $S\subseteq T$ we have $z=(z_t)_{t\in T}\in E$, where $z_t=x_t$ for every $t\in S$ and $z_t=y_t$ for every $t\in T\setminus S$.
   \end{enumerate}}
\end{definition}
Clearly, every invariant under composition set $E$ is invariant under projection for any  $a\in E$.

Following Engelking \cite{Eng1}, M.~Hu\v{s}ek in \cite[p.~132]{Hus} introduced a notion of a $\kappa$-invariant set for $\kappa\ge\aleph_0$ as follows.
\begin{definition}
  {\rm A set $E$ is {\it $\kappa$-invariant}, if for any $x,y\in E$ and $S\subseteq T$ with $|S|<\kappa$ there is a point $z\in E$ such that $z_t=x_t$ for every $t\in S$ and $z_t=y_t$ for every $t\in T\setminus S$.  }
\end{definition}

Developing the above-mentioned concepts of Mazur and Hu\v{s}ek, we introduce the following notions.
 \begin{definition}\label{def:invariant-new}
  {\rm Let $\aleph_i$ and $\aleph_j$ be infinite cardinals, $E\subseteq X$ and $a\in E$. Then $E$ is called
  \begin{enumerate}
    \renewcommand{\theenumi}{\alph{enumi}}
  \renewcommand{\labelenumi}{\theenumi)}
  \item {\it $\aleph_i$-invariant  with respect to the point $a$}, if $x_S^a\in E$ for every $x\in E$ and $S\subseteq T$ with $|S|<\aleph_i$;

 \item  {\it $(\aleph_i,\aleph_j)$-invariant with respect to the point $a$}, if $x_{S_1}^a\in E$ and $x_{T\setminus S_2}^a\in E$ for any point $x\in E$ and for any sets $S_1,S_2\subseteq T$ with $|S_1|<\aleph_i$ and $|T\setminus S_2|<\aleph_j$.
  \end{enumerate}}
\end{definition}
Obviously,   every  $(\aleph_i,\aleph_j)$-invariant set with respect to $a$ is $\aleph_i$-invariant  with respect to $a$.

\begin{definition}
   {\rm A topological space $X$ is said to be
   \begin{itemize}
   \item {\it pseudo-$\aleph_1$-compact}, if any locally finite family of open subsets of $X$ is at most countable;

   \item {\it hereditarily pseudo-$\aleph_1$-compact}, if each subspace of $X$ is pseudo-$\aleph_1$-compact.
   \end{itemize}}
 \end{definition}

It is easy to check that continuous mappings preserve the pseudo-$\aleph_1$-com\-pact\-ness.

The following theorem gives a characterization of the pseudo-$\aleph_1$-compactness of $\aleph_0$-invariant sets and is an analogue of the similar result of Noble and Ulmer \cite[Corollary 1.5]{NobleUlmer} for products.

\begin{theorem}\label{th:char_pseudo}
  Let $(X_t:t\in T)$ be a family of topological spaces, $X=\prod\limits_{t\in T}X_t$, $a\in X$ and let $E\subseteq X$ be an $\aleph_0$-invariant set with respect to $a$. Then the following conditions are equivalent:
  \begin{enumerate}
  \renewcommand{\theenumi}{\roman{enumi}}
    \item  $E$ is pseudo-$\aleph_1$-compact;\label{th:char_pseudo_item1}

    \item for any finite non-empty set $S\subseteq T$ and for any uncountable family $(U_i:i\in I)$ of open sets $U_i$ in $X$ with $U_i\cap E\ne \emptyset$ the family $(p_S(U_i\cap E):i\in I)$ is not locally finite in $p_S(E)$.\label{th:char_pseudo_item2}
  \end{enumerate}
\end{theorem}

\begin{proof}
  $(\ref{th:char_pseudo_item1}) \Rightarrow (\ref{th:char_pseudo_item2})$. Let $S\subseteq T$ be a finite non-empty set, $(U_i:i\in I)$ be an uncountable family of basic open sets $U_i$ in $X$ with $U_i\cap E\ne \O$ and let $V_i=p_S(U_i\cap E)$ for each $i\in I$. If the family $(V_i:i\in I)$ is locally finite in $p_S(E)$, then the family $(p_S^{-1}(V_i)\cap E:i\in I)$ is locally finite in $E$ and $U_i\cap E\subseteq p_S^{-1}(V_i)\cap E$ for each $i\in I$, which contradicts to the  pseudo-$\aleph_1$-compactness of  $E$.

  $(\ref{th:char_pseudo_item2}) \Rightarrow (\ref{th:char_pseudo_item1})$. Consider an uncountable family $(U_i=\prod\limits_{t\in T} U_i^t:i\in I)$ of basic open sets in $X$ such that $U_i\cap E\ne\O$ for all $i\in I$. By \v{S}anin's lemma \cite{Shanin} we choose a finite set $Z$ and uncountable set $J\subseteq I$ such that $N(U_i)\cap N(U_j)=Z$ for all distinct  $i,j\in J$.

  Let $V_i=p_Z(U_i\cap E)$ for all $i\in J$.  It follows from (\ref{th:char_pseudo_item2}) that the family $(V_i:i\in J)$ has a cluster point $v\in p_Z(E)$. Take $y\in E$ such that $v=p_Z(y)$ and put $x=y_Z^a$. We shall show that $x$ is a cluster point of $(U_i\cap E:i\in J)$. Indeed, let $W=\prod\limits_{t\in T}W_t$ be a basic open neighborhood of $x$ in $X$ and  $V=\prod\limits_{t\in Z}W_t\cap p_Z(E)$. Choose such an infinite set $K\subseteq J$ that $V\cap V_i\ne\O$ and $N(W)\cap N(U_i)\subseteq Z$ for all $i\in K$. Take an arbitrary $i\in K$ and a point $b\in V\cap V_i$. Consider a point $c\in U_i\cap E$ with $b=p_Z(c)$ and put $d=c_{Z\cup N(U_i)}^a$. Clearly,  $d\in U_i$. Since $E$ is $\aleph_0$-invariant with respect to $a$ and $c\in E$, $d\in E$. Moreover,  $p_Z(d)=p_Z(c)=b\in V$ and $d_t=a_t\in W_t$ for every $t\in N(W)\setminus Z$. Therefore, $d\in W$. Hence, $d\in W\cap E\cap U_i$.
\end{proof}

The example below shows that the condition (\ref{th:char_pseudo_item2}) in the previous theorem can not be weakened to the following: {\it
the set $p_S(E)$ is pseudo-$\aleph_1$-compact for any non-empty finite set $S\subseteq T$.}

\begin{example}\label{ex:1}
There exists an $(\aleph_1,\aleph_1)$-invariant set $E\subseteq\prod\limits_{t\in T}X_t$ with respect to a point $a\in E$ such that $p_S(E)$ is pseudo-$\aleph_1$-compact for any non-empty finite set  $S\subseteq T$, but $E$ is not pseudo-$\aleph_1$-compact.
\end{example}

\begin{proof}
 Let $T=[0,1]$, $X_0=\mathbb P=\mathbb R\times [0,+\infty)$ be the Niemytzki plane~\cite[p.~21]{Eng}, $X_t=\{0,1\}$ for each $t\in (0,1]$, $X=\prod\limits_{t\in T}X_t$ and let $a=(a_t)_{t\in T}\in X$, where $a_t=0$ for each $t\in (0,1]$ and $a_0=(0,0)$.
 For each $t\in (0,1]$ we define $y^{(t)}=(y^{(t)}_s)_{s\in T}$ and $z^{(t)}=(z^{(t)}_s)_{s\in T}\in X$ as follows:
  \begin{equation*}
 y^{(t)}_s=\left\{\begin{array}{lll}
                         0, & s\in(0,1]\setminus \{t\}\\
                         1, & s=t\\
                         (t,0), & s=0,
                       \end{array}
 \right.
\end{equation*}
\begin{equation*}
 z^{(t)}_s=\left\{\begin{array}{lll}
                         0, & s\in(0,1]\setminus \{t\}\\
                         1, & s=t\\
                         (0,0), & s=0.
                       \end{array}
 \right.
\end{equation*}
Consider the $(\aleph_1,\aleph_1)$-invariant set  $$E=\{y^{(t)}:t\in(0,1]\}\cup \{z^{(t)}:t\in(0,1]\}\cup(X_0\times \prod\limits_{t\in (0,1]}\{0\})$$ with respect to the point $a$. Observe that for any finite set $S\subseteq [0,1]$ the sets $p_S(\{y^{(t)}:t\in(0,1]\})$ and $p_S(\{z^{(t)}:t\in(0,1]\})$ are finite and the set $p_S(X_0\times \prod\limits_{t\in (0,1]}\{0\})$ is separable. Hence, $E$ satisfies the condition mentioned above. But $(\{y^{(t)}\}:t\in (0,1])$ is a locally finite family of open sets in $E$. Therefore,  $E$ is not pseudo-$\aleph_1$-compact.
\end{proof}

\section{Dependence on countably many coordinates of continuous mappings}

\begin{definition}
  {\rm Let $E\subseteq \prod\limits_{t\in T}X_t$. A function $f:E\to Y$  {\it depends on a set $S\subseteq T$} \cite[p.~231]{ComfNerg}, if for all $x,y\in E$ the equality $p_S(x)=p_S(y)$ implies $f(x)=f(y)$. If $|S|\leq \aleph_0$, then we say that $f$ {\it depends on a countably many coordinates}. Similarly, $E$ {\it depends on $S$}, if for all $x\in E$ and $y\in X$ with $p_S(x)=p_S(y)$ we have $y\in E$.}
\end{definition}

\begin{definition}
{\rm  A space $Y$ has {\it a regular $G_\delta$-diagonal} \cite{Zenor}, if there exists a sequence $(G_n)_{n=1}^\infty$ of open subsets of $Y^2$ such that
  \begin{equation}\label{eq:diag}
  \{(y,y):y\in Y\}=\bigcap\limits_{n=1}^\infty G_n=\bigcap\limits_{n=1}^\infty \overline{G_n}.
  \end{equation}}
\end{definition}

We denote $\sigma(a)=\{x\in X:|t\in T:x_t\ne a_t|<\aleph_0\}$ as in \cite{Cors}.

\begin{theorem}\label{th:1imp2}
  Let $Y$ be a space with a regular $G_\delta$-diagonal, $(X_t:t\in T)$ be a family of topological spaces, $X=\prod\limits_{t\in T}X_t$, $a\in X$ and let $E\subseteq X$ be a pseudo-$\aleph_1$-compact subspace which is $(\aleph_1,\aleph_0)$-invariant with respect to $a$. Then for any continuous mapping $f:E\to Y$ there exist a countable set  $T_0\subseteq T$ and a continuous mapping $f_0:p_{T_0}(E)\to Y$ such that $f=f_0\circ (p_{T_0}|_E)$.

In particular, $f$ depends on countably many coordinates.
\end{theorem}

\begin{proof}
Let $(G_n)_{n=1}^\infty$ be a sequence of open sets in $Y^2$ which satisfies~(\ref{eq:diag}) and let  $f:E\to Y$ be a continuous function.  Denote by $T_0$ the set of all
$t\in T$ for which there exist points $x^t,y^t\in E\cap\sigma(a)$ such that
\begin{gather}
x^t_s=y^t_s \mbox{\,\,\,for all \,\,\,} s\ne t,\\
x^t_t=a_t,\label{gath:3}\\
f(x^t)\ne f(y^t).
\end{gather}

Assume that $T_0$ is uncountable and choose an uncountable subset $B\subseteq T_0$ and a number $n_0\in\mathbb N$ such that
\begin{gather*}
  (f(x^t),f(y^t))\in Y^2\setminus \overline{G}_{n_0}\,\,\, \mbox{for all\,\,\,} t\in B.
\end{gather*}
Using the continuity of $f$ at $x^t$ and $y^t$ for every $t\in B$, we find such open basic neighborhoods $U^t$ and $V^t$ of $x^t$ and $y^t$, respectively, that
\begin{gather}
  p_s(U^t)=p_s(V^t)\,\,\,\mbox{for}\,\, s\ne t,\,\,\,\label{gath:2}\\
  f(U^t\cap E)\times f(V^t\cap E)\subseteq Y^2\setminus \overline{G}_{n_0}.\label{gath:1}
\end{gather}

Since $E$ is pseudo-$\aleph_1$-compact and the family $(V^t\cap E:t\in B)$ is uncountable, there exists a point $x^*\in E$ such that for any basic open neighborhood  $W$ of $x^*$ the set $C_W=\{t\in B: V^t\cap E\cap W\ne\O\}$ is infinite.
The continuity of $f$ at $x^*$ implies that there is such a basic neighborhood $W$ of $x^*$  that $f(W\cap E)\times f(W\cap E)\subseteq G_{n_0}$. Notice that $C=C_W\setminus N(W)\ne\O$. Fix $t\in C$ and $y\in V^t\cap E\cap W$. Let $x=y^a_{T\setminus\{t\}}$. Then (\ref{gath:3}) and (\ref{gath:2}) imply that $x\in U^t$. Since $E$ is $(\aleph_1,\aleph_0)$-invariant with respect to $a$, $x\in  E$. Moreover, $x\in W$, since $t\not\in N(W)$. Then $(f(x),f(y))\in G_{n_0}$, which contradicts~(\ref{gath:1}). Hence, the set $T_0$ is countable.

We show that $f$ depends on $T_0$. To do this it is sufficient to check the equality $f(x)=f(x^a_{T_0})$ for every $x\in E$.
Consider the case $x\in E\cap \sigma(a)$. Let $\{t\in T\setminus T_0: x_t\ne a_t\}=\{t_1,\dots,t_m\}$. Then
\begin{gather*}
  f(x)=f(x^a_{T\setminus\{t_1\}})=f((x^a_{T\setminus\{t_1\}})^a_{T\setminus \{t_2\}})=\dots=\\
  =f(((x^a_{T\setminus\{t_1\}})\dots)^a_{T\setminus \{t_m\}})=f(x^a_{T_0}).
\end{gather*}
Now let $x\in E$. Notice that $E\cap \sigma(a)$ is a dense set in $E$. Indeed, if $b=(b_t)_{t\in T}\in E$ and $W$ is a basic open neighborhood of $b$ in  $X$, then $b^a_{N(W)}\in W\cap E\cap\sigma(a)$. Hence, there exists a net $(x_i)$ of points $x_i\in E\cap\sigma(a)$ such that $\lim\limits_i x_i=x$. Then $\lim\limits_i (x_i)^a_{T_0}=x^a_{T_0}$.
It follows from the continuity of $f$ that
$$
f(x)=f(\lim\limits_i x_i)=\lim\limits_i f(x_i)=\lim\limits_i f((x_i)^a_{T_0})=f(\lim\limits_i (x_i)^a_{T_0})=f(x^a_{T_0}).
$$

Consider the function $f_0:p_{T_0}(E)\to Y$ defined by $f_0(z)=f(x)$, if $z=p_{T_0}(x)$ for $x\in E$.
Observe that $f_0$ is defined correctly, because $f$ depends on $T_0$. It remains to prove that  $f_0$ is continuous on $p_{T_0}(E)$. Fix $z\in p_{T_0}(E)$ and a net $(z_i)$ of points $z_i\in  p_{T_0}(E)$ such that $\lim\limits_i z_i=z$. Take  $x\in E$ and $x_i\in E$ with $z=p_{T_0}(x)$ and $z_i=p_{T_0}(x_i)$. Let $y_i=(x_i)^a_{T_0}$ and $y=x^a_{T_0}$. Then $y_i,y\in E$ and $\lim\limits_i y_i=y$. Moreover, since $f$ is continuous at $y$, we have
$$
\lim\limits_i f_0(z_i)=\lim\limits_i f(x_i)=\lim\limits_i f(y_i)=f(y)=f(x)=f_0(z).
$$
Hence, $f_0$ is continuous at $z$.
\end{proof}

Notice that the proof of the dependence of $f$ on $T_0$ in Theorem~\ref{th:1imp2} is similar to the proof of Lemma 2.32 and Lemma 2.27(a) in~\cite{ComfG1}.

\begin{theorem}
  Let $(X_t:t\in T)$ be an uncountable family of topological spaces, $X=\prod\limits_{t\in T}X_t$, $a\in X$ and let $E\subseteq X$ be an $(\aleph_1,\aleph_0)$-invariant set with respect to $a$. Consider the following conditions:
  \begin{enumerate}
  \renewcommand{\theenumi}{\roman{enumi}}
    \item $E$ is pseudo-$\aleph_1$-compact;\label{enum11}

    \item for any space $Y$ with a regular $G_\delta$-diagonal and for any continuous mapping  $f:E\to Y$ there exist a countable set $T_0\subseteq T$ and a continuous mapping $f_0:p_{T_0}(E)\to Y$ such that $f=f_0\circ (p_{T_0}|_E)$;\label{enum12}

    \item for any continuous function $f:E\to \mathbb R$ there exist a countable set $T_0\subseteq T$ and a continuous mapping $f_0:p_{T_0}(E)\to \mathbb R$ such that $f=f_0\circ (p_{T_0}|_E)$.\label{enum13}
  \end{enumerate}
  Then (\ref{enum11}) $\Rightarrow$ (\ref{enum12}) $\Rightarrow$ (\ref{enum13}).

  If $E$ is completely regular and
  \begin{enumerate}\setcounter{enumi}{3}
  \renewcommand{\theenumi}{\roman{enumi}}
   \item for any non-empty open set $U$ in $E$ there exists an uncountable set $T_U\subseteq T$ such that for every $t\in T_U$ there are  $y^{(t)}=(y^{(t)}_s)_{s\in T}, z^{(t)}=(z^{(t)}_s)_{s\in T}\in U$ with $y^{(t)}_t\ne z^{(t)}_t$ and $y^{(t)}_s= z^{(t)}_s$ for every $s\in T\setminus\{t\}$,\label{enum14}
  \end{enumerate}
   then (\ref{enum13}) $\Rightarrow$ (\ref{enum11}).
\end{theorem}

\begin{proof}
   The implication (\ref{enum11}) $\Rightarrow$ (\ref{enum12}) follows from Theorem~\ref{th:1imp2}.

   The implication (\ref{enum12}) $\Rightarrow$ (\ref{enum13}) is obvious.

   Prove that (\ref{enum13}) $\Rightarrow$ (\ref{enum11}). Suppose that $E$ is not pseudo-$\aleph_1$-compact and choose a locally finite in $E$ family $(U_\alpha:\alpha<\omega_1)$ of non-empty open sets $U_\alpha$. Note that $U_\alpha$ may be taken to be disjoint. Indeed, let $(V_i:i\in I)$ be a locally finite family of non-empty open subsets of $E$ with $|I|>\aleph_0$.
   For every $i\in I$ we choose a non-empty open set $W_i\subseteq V_i$ and a finite set $J_i\subseteq I$ such that $W_i\subseteq\bigcap\limits_{j\in J_i}V_j$ and
   $W_i\cap V_j=\emptyset$ for all $j\in I\setminus J_i$. Since $i\in J_i$ for every $i\in I$, $\bigcup\limits_{i\in I}J_i=I$. Now we take a uncountable set $I_0\subseteq I$ such that all the sets $J_i$ from the family $(J_i:i\in I_0)$ are different. Then the uncountable family $(W_i:i\in I_0)$ consists of mutually disjoint elements.

   Since $E$ is completely regular, we may assume that all the sets $U_\alpha$ are functionally open. For every $\alpha<\omega_1$ take a continuous function $f_\alpha:E\to [0,1]$ such that $U_\alpha=f_\alpha^{-1}((0,1])$. Since $T_{U_\alpha}$ is uncountable, we may construct a family $(t_\alpha:\alpha<\omega_1)$ of distinct points $t_\alpha\in T_{U_\alpha}$. According to~(\ref{enum14})  we choose  for every $\alpha<\omega_1$ points $y^{(\alpha)}=(y^{(\alpha)}_s)_{s\in T}, z^{(\alpha)}=(z^{(\alpha)}_s)_{s\in T}\in U_\alpha$ such that $y^{(\alpha)}_{t_\alpha}\ne z^{(\alpha)}_{t_\alpha}$ and $y^{(\alpha)}_s= z^{(\alpha)}_s$ for every $s\in T\setminus\{t_\alpha\}$. Now for every $\alpha<\omega_1$ we choose a continuous function $g_\alpha:E\to [0,1]$ such that $g_\alpha(y^{(\alpha)})=1$ and $g_\alpha(z^{(\alpha)})=0$.

    Consider the continuous function $f:E\to [0,1]$, $f(x)=\sum\limits_{\alpha<\omega_1}f_\alpha(x)g_\alpha(x)$.
    Since the sets $U_\alpha$ are mutually disjoint,
    \begin{gather*}
    f(y^{(\alpha)})-f(z^{(\alpha)})=f_\alpha(y^{(\alpha)})g_\alpha(y^{(\alpha)})-f_\alpha(z^{(\alpha)})g_\alpha(z^{(\alpha)})=f_\alpha(y^{(\alpha)})>0.
    \end{gather*}
    Hence, $f(y^{(\alpha)})\ne f(z^{(\alpha)})$ for every $\alpha<\omega_1$. Since the set $\{t_{\alpha}:\alpha<\omega_1\}$ is uncountable, the function $f$ does not satisfy~(\ref{enum13}).
\end{proof}

\section{Functionally measurable sets}

\begin{proposition}\label{pr:BaireAleph}
   Let $E$ be a subset of a topological space $X=\prod\limits_{t\in T}X_t$ such that for any continuous function $f:E\to \mathbb R$ there exist a countable set $T_0\subseteq T$ and a continuous mapping $f_0:p_{T_0}(E)\to \mathbb R$ with $f=f_0\circ (p_{T_0}|_E)$ and let $0\le\alpha<\omega_1$. Then for any set $A$ of the functionally additive (multiplicative) class $\alpha$ in $E$ there exists a countable set $T_0\subseteq T$ such that $A$ depends on $T_0$ and $p_{T_0}(A)$ is of the functionally additive (multiplicative) class $\alpha$ in $p_{T_0}(E)$.
\end{proposition}

\begin{proof} Let $\alpha=0$. We consider the case when a set $A$ is functionally open in $E$.  Then $A=f^{-1}((0,+\infty))$ for some continuous function $f:E\to \mathbb R$. Take a countable set $T_0\subseteq T$ and a continuous mapping $f_0:p_{T_0}(E)\to \mathbb R$ with $f=f_0\circ (p_{T_0}|_E)$. Then the set $p_{T_0}(A)=f_0^{-1}((0,+\infty))$ is functionally open in $p_{T_0}(E)$. Moreover, if $x\in A$ and $y\in E$ with $p_{T_0}(x)=p_{T_0}(y)$, then $f(y)=f(x)>0$. Therefore, $y\in A$, which implies that $A$ depends on $T_0$.

Assume that the proposition is true for all $\alpha<\beta$ and consider a set $A$ of the functionally additive class $\alpha$ in $E$. Then $A=\bigcup\limits_{n=1}^\infty A_n$, where $A_n$ is of the functionally multiplicative class $\alpha_n<\alpha$ for every $n$. By the assumption for every $n$ there exists a countable set $T_n\subseteq T$ such that $A_n$ depends on $T_n$ and $p_{T_n}(A_n)$ belongs to the functionally  multiplicative  class $\alpha_n$ in $p_{T_n}(E)$. Notice that the set $p_{T_0}(A_n)$ is of the functionally  multiplicative  class $\alpha_n$ in $p_{T_0}(E)$ for every $n$. Then $p_{T_0}(A)=\bigcup\limits_{n=1}^\infty p_{T_0}(A_n)$ is of the functionally additive class $\alpha$ in $p_{T_0}(E)$.
\end{proof}

\begin{definition}
 {\rm Let $0\le\alpha<\omega_1$. A space $X$ is  {\it $\alpha$-universal}, if any subset of $X$ is $\alpha$-embedded in $X$.}
\end{definition}

Clearly, every perfectly normal space is $\alpha$-universal for any $\alpha<\omega_1$.

\begin{proposition}\label{pr:UnivSub}
   Let $0\le\alpha<\omega_1$, $(X_t)_{t\in T}$ be a family of topological spaces such that every countable subproduct is $\alpha$-universal, $X=\prod\limits_{t\in T}X_t$ and let $E\subseteq X$ be such a set as in Proposition~\ref{pr:BaireAleph}. Then $E$ is an $\alpha$-embedded set in~$X$.
\end{proposition}

\begin{proof}
Let $A\subseteq E$ be a set of the functionally multiplicative class $\alpha$ in $E$. According to Proposition~\ref{pr:BaireAleph} there exist a countable set $T_0\subseteq T$ such that  $A$ depends on $T_0$ and $A_0=p_{T_0}(A)$ is of the functionally multiplicative class $\alpha$ in $E_0=p_{T_0}(E)$.  Since  $X_0=\prod\limits_{t\in T_0}X_t$ is $\alpha$-universal, the set $E_0$ is $\alpha$-embedded in $X_0$. Hence, there exists a set $B_0$ of the functionally multiplicative class $\alpha$ in $X_0$ such that $B_0\cap E_0=A_0$. Let $B=p_{T_0}^{-1}(B_0)$. Then $B$ is of the functionally multiplicative class $\alpha$ in $X$, because the mapping $p_{T_0}$ is continuous. Moreover, it is easy to see that $B\cap E=A$.
\end{proof}

\begin{proposition}\label{pr:alphaemb}
  Let $0\le\alpha<\omega_1$, $X=\prod\limits_{t\in T}X_t$ be a pseudo-$\aleph_1$-compact space, where $(X_t)_{t\in T}$ is a family of spaces such that every countable subproduct is $\alpha$-universal and hereditarily pseudo-$\aleph_1$-compact. Then any functionally measurable set  $E\subseteq X$ is $\alpha$-embedded in~$X$.
\end{proposition}

\begin{proof}
 Consider a functionally measurable set $E\subseteq X$. Without loss of generality, we may assume that $E$ belongs to the functionally multiplicative class $\beta$ for some $0\le\beta<\omega_1$. Take a function $f\in B_\beta(X)$ such that $E=f^{-1}(0)$. Since $X$ is pseudo-$\aleph_1$-compact, Theorem~2.3 from \cite{NobleUlmer} implies that there exists s countable set $T_0\subseteq T$ such that for all $x\in E$ and $y\in X$ the equality $p_{T_0}(x)=p_{T_0}(y)$ implies that $y\in E$. Let $E_0=p_{T_0}(E)$. Then
 $$
 E=E_0\times\prod\limits_{t\in T\setminus T_0}X_t.
 $$
Since $\prod\limits_{t\in T_0\cup S} X_t$ is a hereditarily pseudo-$\aleph_1$-compact space, $E_0\times \prod\limits_{t\in S} X_t$ is pseudo-$\aleph_1$-compact space for any finite set $S\subseteq T\setminus T_0$. Hence, by \cite[Corollary 1.5]{NobleUlmer} the set $E$ is pseudo-$\aleph_1$-compact. Therefore, $E$ satisfy the condition of Proposition~\ref{pr:BaireAleph} by Theorem~\ref{th:1imp2} applied to the whole product $E_0\times\prod\limits_{t\in T\setminus T_0}X_t$. It remains to use Proposition~\ref{pr:UnivSub}.
\end{proof}

The following result implies a positive answer to Question~8.1 from~\cite{KarlovaCMUC}.

\begin{corollary}
Let $(X_t)_{t\in T}$ be a family of separable metrizable spaces. Then every functionally measurable subset of $X=\prod\limits_{t\in T} X_t$ is  $\alpha$-embedded in $X$ for any  $0\le\alpha<\omega_1$.
\end{corollary}

\begin{proof}
It follows from Proposition~\ref{pr:alphaemb} and the fact that any countable product of separable metrizable spaces is separable and metrizable, consequently, $\alpha$-universal and hereditarily pseudo-$\aleph_1$-compact.
\end{proof}

\section{The construction of $\alpha$-embedded sets}

\begin{theorem}
For every $0\le\alpha<\omega_1$ there exist a completely regular space $X$ with an $(\alpha+1)$-embedded subspace $E\subseteq X$ which is not $\alpha$-embedded.
\end{theorem}

\begin{proof}
Fix $\alpha<\omega_1$. Let $X_0=[0,1]$, $X_t=\mathbb N$ for every $t\in (0,1]$, $Y=\prod\limits_{t\in (0,1]}X_t$ and
$X=[0,1]\times Y=\prod\limits_{t\in [0,1]}X_t$.

According to~\cite[p.~371]{Ku1} there exists a set $A_1\subseteq [0,1]$ of the additive class $\alpha$ which does not belong to the multiplicative class $\alpha$. Let $A_2=[0,1]\setminus A_1$.

For $i=1,2$ put
$$
F_i=\bigcap\limits_{n\ne i}\{y=(y_t)_{t\in (0,1]}\in Y: |\{t\in (0,1]: y_t=n\}|\le 1\}.
$$
It is easy to see that $F_1$ and $F_2$ are closed disjoint subsets of $Y$.
Let  $B_i=A_i\times F_i$ for $i=1,2$ and $E=B_1\cup B_2$.
Then $B_1$ and $B_2$ are disjoint closed subsets of $E$.

\bigskip
{\it Claim 1.} {\it The set $B_i$ is $\alpha$-embedded in $X$ for every $i=1,2$.}

{\it Proof.} We show that $B_1$ is pseudo-$\aleph_1$-compact (for the set $B_2$ we argue completely similarly). Since $A_1$ is separable, it is enough to check that $F_1$ is pseudo-$\aleph_1$-compact. Notice that the set $F_1$ is $(\aleph_1,\aleph_1)$-invariant with respect to the point $a=(a_t)_{t\in(0,1]}$, where $a_t=1$ for every $t\in(0,1]$. Since for any finite set $S\subseteq (0,1]$ the space $\prod\limits_{t\in S}X_t$ is countable, the set $F_1$ satisfies condition~(\ref{th:char_pseudo_item2}) of Theorem~\ref{th:char_pseudo}. Then by Theorem~\ref{th:char_pseudo} the set $F_1$ is pseudo-$\aleph_1$-compact.

Now observe that each set $B_i$ is $(\aleph_1,\aleph_1)$-invariant with respect to the point \mbox{$a^i=(a_t^i)_{t\in [0,1]}$}, where $a_t^i=i$ for all $t\in (0,1]$ and $a_0^i\in A_i$. It remains to apply Theorem~\ref{th:1imp2} and Proposition~\ref{pr:UnivSub}.

\bigskip
{\it Claim 2.} {\it  The set $E$ is not $\alpha$-embedded in $X$.}

{\it Proof.} Assume the contrary and choose a set $H$ of the functionally multiplicative class $\alpha$ in $X$ such that $H\cap E=B_1$. It follows from Proposition~\ref{pr:BaireAleph}  that there is a countable set $S=\{0\}\cup T$, where $T\subseteq (0,1]$, such that $H$ depends on $S$. Let $y_0\in Y$ be such a point that $p_T(y_0)$ is a sequence of distinct natural numbers which are not equal to $1$ or $2$. Take $y_1\in F_1$ and $y_2\in F_2$ with $p_T(y_0)=p_T(y_1)=p_T(y_2)$. Then for all $x\in A_1$ we have $(x,y_1)\in H$ and, consequently, $(x,y_0)\in H$. Moreover, for all $x\in A_2$ we have $(x,y_2)\not\in H$ and, consequently, $(x,y_0)\not\in H$. Hence, $A_1\times\{y_0\}=([0,1]\times\{y_0\})\cap H$. Therefore, $A_1\times \{y_0\}$ is of the functionally multiplicative class $\alpha$ in $X$, which implies that the set $A_1$ belongs to the functionally multiplicative class $\alpha$ in $[0,1]$, a contradiction.

\bigskip
{\it Claim 3.} {\it The set $E$ is $(\alpha+1)$-embedded in $X$.}

{\it Proof.} Let $C$ be a set of the functionally multiplicative class $(\alpha+1)$ in $E$. Denote
$E_i=A_i\times Y$ for $i=1,2$. Then $E_1$ is of the functionally additive class $\alpha$ and $E_2$ is of the functionally multiplicative class $\alpha$ in~$X$.
For $i=1,2$ put $C_i=C\cap B_i$. Since each of the sets $C_i$ is of the functionally multiplicative class $(\alpha+1)$ in the $\alpha$-embedded set $B_i$ in $X$, there exists a set $D_i$ of the functionally multiplicative class $(\alpha+1)$ in  $X$ such that $D_i\cap
B_i=C_i$.  Let $D=(D_1\cap E_1)\cup (D_2\cap E_2)$. Then $D$ is a set of the functionally multiplicative class $(\alpha+1)$ in $X$ and $D\cap E=C$.
\end{proof}

Notice that the sets $F_i$ were first defined by A. Stone \cite{Stone1} in his proof of non-normality of the uncountable power $\mathbb N^\tau$ of the space $\mathbb N$ of natural numbers.

\section*{Acknowledgement}
The authors express gratitude to the referees for careful reading and comments that helped to improve the paper.

\end{document}